\documentclass{amsart}
\usepackage[english]{babel}
\usepackage{amsfonts, amsmath, amsthm, amssymb,amscd,indentfirst}

\numberwithin{equation}{section}

\usepackage{tabls}

\usepackage{marginnote}

\usepackage{mathrsfs}
\usepackage{euscript}

\usepackage{xcolor}
\usepackage{amsmath,amssymb,latexsym,indentfirst}
\usepackage{times}
\usepackage{palatino}

\usepackage{stmaryrd}

\usepackage{wasysym}

\usepackage{graphicx}
\graphicspath{ {./images/}}

\theoremstyle{plain}
\newtheorem{theorem}{Theorem}[section]

\newtheorem{corollary}[theorem]{Corollary}
\theoremstyle{definition}
\newtheorem{definition}[theorem]{Definition}
\theoremstyle{question}
\newtheorem{question}[theorem]{Question}

\theoremstyle{remark}
\newtheorem{remark}[theorem]{Remark}

\newcommand{\interior}[1]{%
	{\kern0pt#1}^{\mathrm{o}}%
}

\begin{document}
	
	\title[Conserved quantities in General Relativity]{Conserved quantities in General Relativity: the case of initial data sets with a noncompact boundary}
	\author{Levi Lopes de Lima}
	\address{Universidade Federal do Cear\'a (UFC),
		Departamento de Matem\'{a}tica, Campus do Pici, Av. Humberto Monte, s/n, Bloco 914, 60455-760,
		Fortaleza, CE, Brazil.}
	\email{levi@mat.ufc.br}
	\thanks{This essay is based on a talk delivered at IST/Lisboa, in February/2020. The author would like to thank J. Nat\'ario for the invitation and the CAMGSD (Centro de An\'alise Matem\'atica, Geometria e Sistemas Din\^amicos) for the financial support. Also, the author has benefited from support coming from CNPq/Brazil grant 312485/2018-2 and  FUNCAP/CNPq/PRONEX grant 00068.01.00/15.}
	
	\begin{abstract}
		It is well-known that considerations of symmetry lead to the definition of a host of conserved quantities (energy, linear momentum, center of mass, etc.) for an asymptotically flat initial data set, and a great deal of progress in Mathematical Relativity in recent decades essentially amounts to establishing fundamental properties for such quantities (positive mass theorems, Penrose inequalities, geometric representation of the center of mass by means of isoperimetric foliations at infinity, etc.) under suitable energy conditions. 
	In this article I first review certain aspects of this classical theory and then describe how they can be (partially) extended to the setting in which the initial data set carries a non-compact boundary. In this case, lower bounds for the scalar curvature in the interior and for the mean curvature along the boundary both play a key role. Our presentation aims to highlight various rigidity/flexibility phenomena coming from the validity, or lack thereof, of the corresponding positive mass theorems and/or Penrose inequalities.
	\end{abstract}

	\maketitle
	\tableofcontents
	
	\section{Introduction}\label{Intro}
	
	The ADM formalism in General Relativity has provided a systematic way to attach to certain solutions of Einstein field equations an array of conserved quantities, thus resolving, at least in the special cases where the strategy succeeds, an old controversy related to the notorious difficulties in making sense of such invariants.
	The underlying idea consists in requiring that in the asymptotic region the corresponding initial data set decays in a suitable sense towards some reference solution that, besides enjoying other nice properties which are irrelevant for the present discussion, is assumed to carry a {\em nontrivial} isometry group. After transplanting the corresponding Killing vector fields to the given initial data set by means of the chosen identification at infinity, one is able to apply our favorite rendition of Noether's principle relating symmetries to conservation laws in order to exhibit the desired quantities, which in general are expressed in terms of certain flux integrals over the boundary at infinity.  
	It then follows that these quantities are  conserved under time evolution of the system whenever the matter fields decay fast enough (in particular, for vacuum solutions).
	In addition to playing a paramount role in the understanding of the dynamics of solutions, the study of theses quantities (notably the energy and the center of mass) reveals deep connections with  Geometric Analysis (positive mass theorems and their applications to the Yamabe problem, Penrose inequalities and the inverse mean curvature flow, geometric representation of the center of mass by means of isoperimetric foliations at infinity, etc.). In this setting, the scalar curvature emerges as the fundamental concept linking together these quite disparate realms, as it not only may be viewed as the energy density along initial data sets but also appears prominently in the zero order term of the relevant differential operators (the conformal Laplacian, the Jacobi operator on minimal surfaces, the Dirac Laplacian acting on spinors, etc.). The interactions between General Relativity and  Geometric Analysis arising from these ideas are countless and the aim of this survey is to convey certain aspects of this narrative from a rather personal perspective, with an emphasis towards positive mass theorems and the various rigidity/flexibility phenomena stemming from them. 
	
	We start our journey in Section \ref{class} by recalling  how the classical attempts to make sense of the total energy of a gravitational system in General Relativity were incorporated into the  ADM approach to Canonical Gravity \cite{arnowitt1962gravitation} so as to allow  a precise formulation of the so-called Positive Mass Theorem (PMT) for initial data sets meeting suitable energy conditions. A proof of this central result was eventually obtained by Schoen and Yau  using minimal surfaces \cite{schoen1981energy} and soon afterwards Witten presented an alternate proof based on spinors \cite{witten1981new}. Although the spin assumption is quite restrictive in higher dimensions, Witten's method revealed itself quite flexible, being particularly useful when trying 
	to extend this circle of ideas to initial data sets carrying a noncompact boundary \cite{almaraz2014positive,almaraz2020mass, almaraz2021spacetime}. A distinctive feature of this latter set of results, which is reviewed in Section \ref{cons:qt:bd}, is that besides a lower bound on the scalar curvature in the interior we also require, quite naturally, a lower bound on the mean curvature along the boundary\footnote{In General Relativity, this intimate relationship between scalar and mean curvatures makes its debut in the  Gibbons-Hawking-York action, which plays a central role in the so-called path integral approach to Quantum Gravity \cite{york1972role,gibbons1977action}; see (\ref{gibb:haw}) below. For an in-depth account of the geometric side of this story, see \cite{gromov2019four}.}. Similarly to what happens in the boundaryless case, the (time-symmetric) PMT presented in Theorem \ref{main:abl} also finds notable applications in Geometric Analysis, especially in regard to the Yamabe problem on manifolds with boundary \cite{almaraz2010existence,almaraz2015convergence,almaraz2019compactness}, besides inspiring further developments \cite{koerber2019riemannian}. 
	In Section \ref{non:zero:lamb}, certain rigidity/flexibility phenomena are summarized in case the initial data set yields a solution with  {\em nonzero} cosmological constant. In the asymptotically hyperbolic case we present a rigidity result for conformally compact Einstein manifolds carrying a {\em minimal} inner boundary obtained in \cite{almaraz2020mass}, which may be of some interest in connection with the so-called AdS/BCFT correspondence. Also, we discuss a flexibility result for the de Sitter-Schwarzschild metric \cite{cruz2018deforming} in the line of the famous (negative) solution of Min-Oo's conjecture by Brendle-Marques-Neves \cite{brendle2011deformations}. 
	In Section \ref{center:sec}, we complete our survey by explaining how the well-known  connections between the center of mass and isoperimetry in the boundaryless 
	setting  ensuing from the seminal article by Huisken and Yau \cite{huisken1996definition}  also admit suitable extensions in the presence of a boundary \cite{almaraz2020center}.  Finally, with a view towards the future we intersperse along the text a few interesting problems in this area of research.

	\vspace{0.3cm}
	\noindent
	{\bf Acknowledgments.} I heartily thank
	S. Almaraz, C.T. Cruz, F. Gir\~ao, L. Mari and J. Nat\'ario for reading a preliminary version of this article and contributing with valuable suggestions. 
	
	\section{The classical legacy}\label{class}
	
	The scalar curvature has played a significant role in General Relativity (GR) since its inception by A. Einstein in 1915. Its use in the  Lagrangian formulation of the theory  was justified by H. Vermeil, a student of F. Klein at G\"ottingen, who checked the allegedly folklore result that the scalar curvature $R_{\overline g}$ is the unique Riemannian/Lorentzian scalar invariant which depends on derivatives of the metric up to second order, being linear in the second derivatives. This naturally led to the selection of the so-called {\em Hilbert-Einstein action}
	\begin{equation}\label{action:eh}
		(\overline g,\psi)\mapsto \int_{\overline M} \left(R_{\overline g}+\eta\mathscr T_{\overline g,\psi}\right)d{\overline M},
	\end{equation}   
	where $\overline g$ is a Lorentzian metric on a given $4$-manifold\footnote{Even though most of the results in Sections \ref{class}, \ref{cons:qt:bd} and \ref{center:sec} hold true in any spacetime dimension $n+1\geq 4$, we momentarily work in the physical dimension $n=3$ in order to simplify the exposition. Also, we assume that all Lorentzian manifolds are time-oriented. Moreover, 
		when representing tensor quantities in local coordinates, we will make use of the index ranges  $0\leq\alpha,\beta,\ldots \leq 3$ and $1\leq, i,j,\ldots\leq 3$, with the convention that the label $0$ is reserved for a chosen time coordinate.} 
	$\overline M$ and $\mathscr T$ is the matter-energy Lagrangian density, which is assumed to depend on the metric and on some external matter field $\psi$\footnote{For our purposes, it suffices to assume that $\psi$ varies among the sections of a natural vector bundle over $\overline M$, so the action (\ref{action:eh}) is generally covariant, i.e. preserved by the action of the diffeomorphism group of $\overline M$. This choice already includes many interesting examples.}. By extremizing this action with respect to the metric and discarding boundary terms we obtain {\em Einstein field equations}
	\begin{equation}\label{field:eq}
		-G_{\overline  g}+8\pi T_{\overline g,\psi}=0,
	\end{equation}
	whose solutions $(\overline g,\psi)$ encode the dynamical features of the theory. Here, \[
	G_{\overline g}={\rm Ric}_{\overline g}-\frac{R_{\overline g}}{2}\overline g
	\] is the Einstein tensor of $\overline g$, $T_{\overline g,\psi}$ is the stress-energy-momentum tensor and we have adjusted the coupling constant $\eta$ conveniently.
	We should also add to (\ref{field:eq}) the system of equations $U_{\overline g,\psi}=0$ obtained by extremizing (\ref{action:eh}) with respect to $\psi$, but this should not concern us here.

	The problem remains of 
	extracting physical information out of the highly nonlinear system (\ref{field:eq}).
	This already afflicted the founding fathers of GR, especially in regard to the status of energy conservation in the theory. In a by now well documented story, both Hilbert and Klein commissioned the algebraist E. Noether to clarify the role played by certain differential identities satisfied by the variational derivative of the action \cite{kosmann2012noether}. In a stroke of genius,  
	Noether derived these identities in the general framework of her 
	celebrated Second Theorem, which applies, for instance, to any Lagrangian theory whose symmetry group depends locally on  finitely many functions of the spacetime variables. 
	In the specific setting of GR as formulated above, these identities read as  
	\begin{equation}\label{cons:noe}
		{\rm div}_{\overline g}\left(-G_{\overline  g}+8\pi T_{\overline g,\psi}\right)-\frac{1}{2}\mathcal L\psi^*U_{\overline g,\psi}=0,
	\end{equation}  
	where $\mathcal L\psi^*$ is the formal adjoint of the Lie derivative $\mathcal L\psi$ acting on vector fields \cite{barbashov1983continuous}.  
	We stress that this holds true for {\em any} pair $(\overline g,\psi)$, irrespective of it being a solution. Notice also that in vacuum (absence of matter fields) this reduces to the contracted Bianchi  identities for an arbitrary metric $\overline g$.  
	As explained elsewhere, the aftereffect of this discussion is that even though GR admits a huge symmetry group encompassing the diffeom\-orphisms of $\overline M$, none of these yields a {\em nontrivial} conservation law on a given solution  by the standard procedure \cite{olver2000applications}.
	In this way, full covariance, which is a treasured feature of the theory and happens to be the  mechanism behind the validity of (\ref{cons:noe}), actually entails the rather disappointing conclusion that conservation laws in GR are distinct in nature from those appearing for instance in Classical Mechanics or Special Relativity, which lie both in the confines of Noether's First Theorem and hence are amenable to the more conventional treatment \footnote{The notable exception here takes place in the rather special cases where the underlying space-time carries Killing vector fields; this covers the stationary case, for instance \cite{beig1978arnowitt}.}.

	Despite this rather discouraging outcome, the formalism above already contains the clue towards an alternate construction of {\em nontrivial} conserved quantities in GR, at least for a restricted class of solutions. The key observation is that the differential identities in (\ref{cons:noe}) suggest that only six out of the ten equations in (\ref{field:eq}) actually carry dynamical content, whereas the remaining four equations correspond to constraints relating the initial
	values of the fields instead of determining how these fields evolve. The standard way to confirm this intuition is to consider a spacelike slice $(M,g,h)\hookrightarrow (\overline M,\overline g)$, which means that $M$ is an embedded spacelike hypersurface, $g=\overline g|_M$ is the induced Riemannian metric and $h$ is the associated second fundamental form (with respect to the future directed, timelike unit normal vector field to $M$). If we set  
	\[
	\mu(g,h)=\frac{1}{16\pi}\left(R_g-|h|_g^2+({\rm tr}_gh)^2\right), \quad J(g,h)=\frac{1}{8\pi}\left({\rm div}_gh-d{\rm tr}_gh\right),
	\]
	then the Gauss and Codazzi equations of hypersurface theory applied to (\ref{field:eq}) yield the {\em constraint equations}
	\begin{equation}\label{const:eq}
		\mu=T_{00}, \quad J_i=T_{0i}.
	\end{equation}    
	Interestingly enough, we see here the scalar curvature $R_g$ of $g$ resurfacing as a multiple of the energy density as measured by an observer comoving with the slice in the time-symmetric case ($h=0$). 
	
	A major breakthrough in Mathematical Relativity took place when Mme. Y. Choquet-Bruhat reversed this line of thought and proved that, at least in vacuum, (\ref{const:eq}) is a sufficient condition for the existence of solutions. More precisely, she showed that if one is given a triple $(M,g,h)$, where $(M,g)$ is a Riemannian $3$-manifold and $h$ is a twice covariant symmetric tensor on $M$ satisfying (\ref{const:eq}) with $T=0$ then there exists a Lorentzian $4$-manifold $(\overline M,\overline g)$ satisfying $G_{\overline g}=0$ and a spacelike isometric embedding $(M,g)\hookrightarrow (\overline M,\overline g)$ such that $h$ is the induced second fundamental form. Roughly, this is proved by propagating in time the given {initial data} $(M,g,h)$ by means of the evolution system corresponding to the $(i\leq j)$-components of the field equations and then checking that the constraints (in this case, $\mu=0$ and $J=0$) are preserved; see \cite{ringstrom2015origins,choquet2015beginnings} for recent accounts of her work, including extensions by herself and others in the presence of matter fields. 
	
	Choquet-Bruhat's theorem and its variants provided an initial value formulation for GR which shifted the focus from a solution $(\overline M,\overline g)$ to an initial data set $(M,g,h,\mu,J)$, thus launching a whole new perspective on the subject. The following definition, which is an outgrowth of the so-called ADM approach to Canonical Gravity \cite{arnowitt1962gravitation}, illustrates this viewpoint in the context  of the search for conserved quantities for a class of initial data sets modeling isolated gravitational systems. 
	
	\begin{definition}\label{adm:ids}
		An initial data set (IDS) $(M,g,h,\mu,J)$ is asymptotically flat with decay rate $\tau\in (1/2,1]$ if there exists an exterior region $M_{\rm ext}$, with $M\backslash M_{\rm ext}$ compact, and a diffeomorphism $\mathbb R^3\backslash \overline{B_1(0)}\cong  M_{\rm ext}$ such that, in the corresponding asymptotic coordinate $x$,  there hold 
		\[
		g_{ij}(x)=\delta_{ij}+O_3(r^{-\tau}), \quad h_{ij}(x)=O_1(r^{-\tau-1}),
		\]
		and 
		\[
		\mu(x)=O(r^{-2\tau-2}), \quad J(x)=O(r^{-2\tau-2}),
		\]
		as $r=|x|_\delta\to +\infty$.
	\end{definition}

	Intuitively, this means that as one approaches spatial infinity, $(M,g,h,\mu,J)$ converges in a suitable  sense to $(\mathbb R^3,\delta,0,0,0)$, an IDS of Minkowski vacuum space $(\mathbb L^{1,3},\overline \delta, 0)$ of Special  Relativity. Now, since this maximally symmetric vacuum solution carries 
	a $10$-dimensional space of isometries (the Poincar\'e group), we are tempted to use the restrictions of  the corresponding Killing fields to the slice $t=x_0=0$ in order to construct conserved quantities in the asymptotic limit via the appropriate version of Noether's First Theorem \cite{christodoulou2008mathematical,harlow2020covariant}.  This works fine indeed and Table \ref{table:1} provides the final outcome.

\vspace{0.3cm}
	
	\begin{table}	
		\setlength{\tablinesep}{0.5em}{
		\begin{tabular}{c| c| c}
				Killing vector field & Conserved quantity  & Surface integral  \\
				\hline
				$\partial_0$ (time translation) & energy $E$ & $16\pi E=\int_{S_\infty^2}\mathbb U_{\bf 1}(\nu) dS^2_\infty$ \\
				$\partial_i$ (spatial translation) & 	linear momentum $P$ & $8\pi P_i=\int_{S_\infty^2}({\bf i}_{\partial_i}\Pi)(\nu) dS^2_\infty$  \\
				$x_i\partial_0+x_0\partial_i$ (boost) &  center of mass $C$ & $16\pi E C_i=\int_{S_\infty^2}\mathbb U_{x_i}(\nu) dS^2_\infty$  \\
				$X_{ij}=x_i\partial_j-x_j\partial_i$ (spatial rotation) &  angular momentum $Q$ & $8\pi Q_{ij}=\int_{S_\infty^2}({\bf i}_{X_{ij}} \Pi)(\nu) dS^2_\infty$  \\ 
				\hline
			\end{tabular}}
		\vspace{0.2cm}
		\caption{Conserved quantities for asymptotically flat IDS's}
			\label{table:1}
		\end{table}

\vspace{0.3cm}

	Regarding this table, a few comments are in order:
	\begin{itemize}
		\item We set
		\[
		\mathbb U_w=w({\rm div}_\delta e-d{\rm tr}_\delta e)-{\bf i}_{\nabla_\delta w} e+({\rm tr}_\delta e)dw,
		\] 
		and 
		\[
		\Pi=h-({\rm tr}_gh)g,
		\]
		where $e=g-\delta$, $w:\mathbb R^3\to\mathbb R$ is a function, ${\bf 1}$ represents the function identically equal to $1$ and the infinity symbol means that we first compute the flux integrals over a large coordinate sphere $S_r^2$ and then pass the limit as $r\to+\infty$. In particular, $\nu$ is the outward pointing unit normal vector field to this sphere. 
		\item $E$ and $ C$ only depend on the {\em intrinsic} geometry of the initial data set. Notice also that $ C$ only makes sense if $E\neq 0$.
		\item $(E,P)$ is the energy-momentum vector of the corresponding gravitational system. It is well defined under the ADM decay conditions in Definition \ref{adm:ids}.
		\item $C$ and $Q$ are well defined only if we further assume the so-called {\em Regge-Teitelboim} conditions:
		\[
		g^{\rm odd}(x)=O(r^{-\tau-1}), \quad h^{\rm even}(x)=O(r^{-\tau-2}), 
		\]
		and 
		\[
		\mu^{\rm odd}(x)=O(r^{-2\tau-3}), \quad 	J^{\rm odd}(x)=O(r^{-2\tau-3}).
		\]
		Here, $	f^{\rm odd}(x)=(f(x)-f(-x))/2$ and $f^{\rm even}=f-f^{\rm odd}$ for any function in the asymptotic region.  
		\item These quantities are conserved under time evolution if the matter fields decay fast enough at infinity, except for the center of mass that satisfies $dC/dt=P$, as expected \cite{christodoulou2008mathematical}; this also follows from the quasi-local approach in \cite{wang2016energy}. 
	\end{itemize}
	
	A great deal of progress in Mathematical Relativity in recent decades amounts to establishing fundamental properties of these quantities. The prominent result in this area is the Positive Mass Theorem (PMT) due to R. Schoen and S.-T. Yau. To state it, first we need  single out those initial data sets which are viable from a physical viewpoint. 
	
	\begin{definition}\label{dom:ener}
		An initial data set $(M,g,h,\mu,J)$ satisfies the dominant energy condition (DEC) if there holds 
		\begin{equation}\label{dom:energ:eq}
			\mu\geq |J|_g
		\end{equation}
		everywhere along $M$. 
	\end{definition}  
	
	Note that, by (\ref{const:eq}), this means that the $4$-vector $T^{0\alpha}$ is causal and future directed. 
	In words: a comoving observer never sees the energy-momentum density $(\mu,J)$ flowing faster than light!

	\begin{theorem}\cite{schoen1981energy}\label{pmt}
		If  $(M,g,h,\mu,J)$ is an asymptotically flat IDS satisfying the DEC (\ref{dom:energ:eq}) then $E\geq |P|_\delta$, with the equality holding if and only if $E=|P|_\delta=0$ and $(M,g,h)$  can be isometrically embedded in Minkowski space.   
	\end{theorem}

	Thus, under (\ref{dom:energ:eq}) we see that $(E,P)$ is causal and future directed as a $4$-vector in Minkowski space. In words: viewed from infinity the gravitational system modeled by the IDS moves as a massive particle in Special Relativity. 
	This remarkable result was first proved using minimal surfaces techniques \cite{schoen1981energy} and soon after by using spinors \cite{witten1981new}. We note that the rigidity statement in arbitrary dimensions has been settled in full generality only recently \cite{huang2019equality}; we refer to this latter article for details on the history of this subject.  
	
	From our perspective the next corollary is worth mentioning.
	
	\begin{corollary}\label{cor:pmt}
		Let $(M,g)$ be a time-symmetric, asymptotically flat  IDS (i.e. an asymptotically flat Riemannian manifold) satisfying $R_g\geq 0$ everywhere. Then the ADM mass $m_{ADM}:=E$ is nonnegative and vanishes only if $(M,g)=(\mathbb R^3,\delta)$ isometrically. 
	\end{corollary}
	
	This result, which actually holds true in any dimension $n\geq 3$ \cite{schoen2017positive,lohkamp2016higher}, has notable applications to many topics in Geometric Analysis, including the Yamable problem and its developments \cite{schoen1984conformal,leeparker1987,brendle2015recent}. But perhaps as important as the direct applications themselves is the injection of new methods in a honorable area of research. For instance, the tools employed to establish a rather special case of Corollary \ref{cor:pmt}, namely, the so-called {\em Geroch conjecture}\footnote{This says that a metric with nonnegative scalar curvature on a torus is necessarily flat.}, when combined with the celebrated  ``controlled surgery'' technique, eventually led to a huge collection of remarkable accomplishments, including the classification of closed, simply connected manifolds of dimension $n\geq 5$ carrying a metric with positive scalar curvature \cite{schoen1979structure,gromov1980classification,gromov1983positive,stolz1992simply}.
	
	\section{Conserved quantities in the presence of a noncompact boundary}\label{cons:qt:bd}  
	
	A version of the Yamabe problem on compact manifolds with boundary asks for a conformal metric which is scalar flat in the interior and has constant mean curvature along the boundary \cite{escobar1992conformal}. The strategy of  attack here mimics the boundaryless situation: in the simpler case, the test function to be inserted in the relevant Yamabe quotient is devised by local methods, whereas in the harder case the construction of the test function is global in nature and in principle requires a version of the PMT in the presence of a noncompact boundary \cite{almaraz2010existence}\footnote{See also \cite{almaraz2019compactness} and the references therein for applications of this PMT to the compactness of the space of solutions.}. Even though this latter case was eventually settled by alternate methods \cite{mayer2017barycenter}, the use of a PMT is definitely more conceptual and seems to be indispensable when dealing with a parabolic version of the problem \cite{almaraz2015convergence}. The appropriate version of the PMT applies to the class of Riemannian manifolds described below. 
	
	\begin{definition}\cite{almaraz2014positive}\label{asym:flat:Pi}
		A $3$-manifold $(M,g)$ is asymptotically flat with a non-compact boundary $\Sigma$ if there exists  a diffeomorhism  $\mathbb R^3_{+}\backslash \overline{B_1(0)}\cong M_{\rm ext}$ such that, in the corresponding asymptotic coordinates $x$, there hold 
		\[
		e^+:=g-\delta^+ =O_3(r^{-\tau}), \quad \tau>\frac{1}{2},
		\]
		and 
		\[
		R_g=O(r^{-3-\sigma}), \quad H_g=O(r^{-2-\sigma}), \quad \sigma>0,
		\]
		as $r=|x|_\delta\to +\infty$. Here,  $H_g$ is the mean curvature of $\Sigma$, $\mathbb R^3_+=\{x\in \mathbb R^3; x_3\geq 0\}$ and $\delta^+=\delta|_{\mathbb R^3_+}$.
	\end{definition}      
	
	For this kind of manifold we may define a notion of mass (or energy) $\mathscr E$ as indicated in the first line of Table \ref{table:2} below, 
	where $\partial S^2_{\infty,+}$ is the large scale limit of coordinate hemispheres $S^2_{r,+}$ of radius $r$ centered at the origin of the model space $\mathbb R^3_+$ and $S^1_{\infty}=\partial S^2_{\infty,+}$ is the boundary at infinity of $\Sigma$. However, as in the boundaryless case, this is just the tip of the iceberg!
	Indeed,
	we may picture a situation where the triple $(M,g,\Sigma)$ is isometrically embedded in some Lorentzian manifold $(\overline M,\overline g)$ carrying a $3$-dimensional time-like boundary $\overline \Sigma$ such that $\Sigma=M\cap\overline \Sigma$.  In order to recover this as the IDS associated to a solution of appropriate field equations we must replace (\ref{action:eh}) 
	by the {\em Gibbons-Hawking-York action}
	\begin{equation}\label{gibb:haw}
		(\overline g,\psi,\xi)\mapsto \int_{\overline M} \left(R_{\overline g}+\eta T_{\overline g,\psi}\right)d{\overline M}+\int_{\overline \Sigma}\left(2H_{\overline g}+\hat{\eta}\mathscr S_{g,\xi}\right)d{\overline \Sigma},
	\end{equation}
	where $H_{\overline g}$ is the mean curvature of the embedding $\overline \Sigma\hookrightarrow\overline M$ and $\mathscr S$ is the matter-energy Lagrangian density due to a matter field $\xi$ distributed along $\overline \Sigma$ \cite{york1972role,gibbons1977action}.
	After extremizing this with respect to $\overline g$, but this time carefully taking into account the  boundary terms, we obtain the corresponding field equations: 
	\begin{equation}\label{efe}
		\left\{
		\begin{array}{rcc}
			G_{\overline g}& = &{8\pi}T\,\,\,{\rm in}\,\,\,\overline M,
			\\
			{\overline\Pi}_{\overline g}&=& {8\pi}S\,\,\,\text{on}\,\,\,\overline \Sigma.
		\end{array}
		\right.
	\end{equation}
	Here, $\overline \Pi_{\overline g}=\overline B-H_{\overline g} \overline g|_{\overline \Sigma}$, where $\overline B$ is the second fundamental form of $\overline\Sigma$ with respect to the inward pointing unit normal vector $\varrho$, $S$ is the boundary stress-energy-momentum tensor on $\overline \Sigma$ induced by the matter distribution associated to $\xi$\footnote{{We stress that $\mathscr S$ keeps no relationship with $\mathscr T|_{\overline \Sigma}$}.} and we have set $\hat{\eta}=\eta$. As we have seen, restriction of the first system of equations in (\ref{efe}) to $M$ yields the {\em interior} constraint equations in (\ref{const:eq}). On the other hand, 
	if we further assume  that $M$ meets $\overline \Sigma$ orthogonally along $\Sigma$, then a computation shows that the
	restriction of the second system of equations in (\ref{efe}) to $\Sigma$ gives the {\em boundary constraint equations}\footnote{The reason for the terminology is that, similarly to what happens in (\ref{const:eq}), (\ref{constraint2}) seems to relate the
		initial values of fields along $\Sigma$ instead of determining how these fields evolve. It remains to investigate
		whether these boundary constraints play a role in the corresponding initial-boundary value
		Cauchy problem.}
	\begin{equation}\label{constraint2}
		\left\{
		\begin{array}{rcl}
			H_g & = & {8\pi}S_{00},\\
			({\bf i}_{\varrho}{{\Pi}})_a & = &  {8\pi}S_{0a},
		\end{array}
		\right.
	\end{equation}
	where the indexes $1\leq a,b,\ldots\leq 2$ refer to directions tangential to $\Sigma$. Needless to say, the reader's attention should be drawn to the distinctive role played by the mean curvature in the discussion above.

	Now let $\mathbb L^{1,3}_+=\{x\in\mathbb L^{1,3};x_3\geq 0\}$ be the {\em Minkowski half-space}, whose boundary $\partial \mathbb L^{1,3}_+$ is a time-like hypersurface. Notice that $\mathbb L^{1,3}_+$ carries the totally geodesic spacelike hypersurface $\mathbb R^3_+=\{x\in\mathbb L^{1,3}_+;x_0=0\}$, the {\em Euclidean half-space}, which is endowed with the standard flat metric 
	$\delta=\overline \delta|_{\mathbb R^{3}_+}$. Notice also that $\mathbb R^3_+$  carries a totally geodesic boundary $\partial \mathbb R^3_+$.  
	We now make precise the requirement that the spatial infinity of $\overline M$, as observed along the initial data set $(M,g,h,\Sigma)$, is modeled on the inclusion $\mathbb R^3_+\hookrightarrow \mathbb L^{1,3}_+$.

	\begin{definition}\label{def:as:hyp}\cite{almaraz2021spacetime}
		We say that the IDS $(M,g,h,\Sigma)$ is {{asymptotically flat}} (with a non-compact  boundary $\Sigma$) if there exists a diffeomorphism $\mathbb R^3_+\backslash \overline{B_1(0)}\cong  M_{\rm ext}$ such that, as $r\to +\infty$, 
		\[
		g_{ij}(x)=\delta^+_{ij}+O_3(r^{-\tau}), \quad h_{ij}(x)=O_1(r^{-\tau-1}),
		\]
		\[
		\mu(x)=O(r^{-2\tau-2}), \quad J(x)=O(r^{-2\tau-2}),
		\]
		and 
		\[
		H_g=O(r^{-2-\sigma}), \quad {({\bf i}_\varrho  \Pi)^\top}=O(r^{-2-\sigma}),
		\]
		where ${}^{\top}$ means orthogonal projection onto $\Sigma$. 
	\end{definition}

	We may now proceed as in the boundaryless case and use the Killing vector fields on $\mathbb L^{1,3}_+$ associated to the isometries leaving the boundary invariant\footnote{Note that this ``half-Poincar\'e algebra'' is isomorphic to the full Poincar\'e algebra of $\mathbb L^{1,2}$. Thus, from a physical viewpoint, our model at infinity corresponds to Special Relativity with one less spatial dimension.} in order to define an array of conserved quantities as portrayed in Table \ref{table:2}\footnote{We first learned about the energy $\mathscr E$ in conversations with F. Marques.}. Here, $e^+=g-\delta^+$, $\vartheta$ is the outward co-normal unit
	vector field to $S^1_{r}\hookrightarrow \Sigma$ (computed with respect to $\delta$) and $a,b=1,2$. We also note that the extra flux integrals over $S^1_{\infty}$, which reflect the presence of the boundary $\Sigma$, only appear in the expressions for the energy and the center of mass, which  correspond to those Killing vector fields which are {\em normal} to the IDS.
	
	\vspace{0.3cm}
	
	\begin{table}
		\setlength{\tablinesep}{0.5em}{
		\begin{tabular}
	{  c| c| c } 
	\hline
				Killing vector field & Conserved quantity  & Surface integral  \\
				$\partial_0$ & energy  $\mathscr E$ & $16\pi \mathscr E=\int_{S_{\infty,+}^2}\mathbb U_{\bf 1}(\nu) dS^2_{\infty,+} -\int_{S^1_{\infty}}e^+(\partial_3,\vartheta)dS^1_{\infty}$\\
				$\partial_a$ & 	linear momentum $\mathscr P$ & $8\pi \mathscr P_a=\int_{S_{\infty,+}^2}({\bf i}_{\partial_a} \Pi)(\nu) dS^2_{\infty,+}$  \\
				$x_a\partial_0+x_0\partial_a$ &  center of mass $\mathscr C$ & $16\pi \mathscr E\mathcal \mathscr \mathscr C_a=\int_{S_{\infty,+}^2}\mathbb U_{x_a}(\nu) dS^2_{\infty,+}-\int_{S^1_{\infty}}x_ae^+(\partial_e,\vartheta)dS^1_{\infty}$  \\
				$X_{ab}=x_a\partial_b-x_b\partial_a$ &  angular momentum $\mathscr Q$ & $8\pi \mathscr Q_{ab}=\int_{S_{\infty,+}^2}({\bf i}_{X_{ab}} \Pi)(\nu) dS^2_{\infty,+}$  \\
		\hline
	\end{tabular}}
\vspace{0.2cm}
\caption{Conserved quantities for asymptotically flat IDS's with a noncompact boundary}

\label{table:2}\end{table}

\vspace{0.3cm}

	We may now state the main result in \cite{almaraz2021spacetime}, which provides a PMT for the energy-momentor $3$-vector $(\mathscr E,\mathscr P)$ 
	under suitable DECs on the interior and along the boundary.
	
	\begin{theorem}\label{main}\cite{almaraz2021spacetime}
		Let $(M,g,h,\Sigma)$ be an asymptotically flat IDS with a noncompact boundary as above and assume that it satisfies the interior  DEC
		\begin{equation}\label{dom:energ:eq:0}
			\mu\geq |J|_g\quad\, {\rm in}\,\quad M,
		\end{equation}
		and the boundary 
		DEC
		\begin{equation}\label{dom:energ:bd}
			H_g\geq |({\bf i}_\varrho \Pi)^\top|_g \quad\, {\rm on}\,\quad \Sigma.
		\end{equation}
		Then  $\mathscr E\geq |\mathscr P|_\delta$,
		with the equality holding only if $\mathscr E= |\mathscr P|_\delta=0$ and  
		$(M,g)$ can be isometrically embedded in $\mathbb L^{1,3}_+$ in such a way that  $h$ is the induced second fundamental form, $\Sigma$ is totally geodesic (as a hypersurface in $M$), lies on $\partial\mathbb L^{1,3}_+$ and $M$ is orthogonal to $\partial\mathbb L^{1,3}_+$ along $\Sigma$. 
	\end{theorem}
	
	\begin{proof} (sketch)
		Since $M$ is spin, we may use the DECs to ensure that there exists a spinor $\psi$ on $M$ which is harmonic with respect to the corresponding Dirac-Witten operator and which satisfies suitable boundary conditions both at infinity (it asymptotes a previously chosen ``constant'' spinor  $\phi$ in the model $\mathbb L^{1,3}_+$) and along the boundary (it satisfies the so-called MIT bag boundary condition). With these spinors at hand,  a somewhat involved computation establishes a nice extension of Witten's celebrated formula \cite{witten1981new} for the energy-momentum vector in the presence of a boundary:
		\begin{eqnarray*}
			\frac{1}{4}\left(\mathscr E|\phi|^2-\langle \phi,\mathscr P_a \partial_{x_0}\cdot \partial_{x_a}\cdot\phi\rangle \right) 
			& = & \int_{M}(|\overline\nabla\psi|^2+\langle \mathcal R\psi,\psi\rangle)d{M}\nonumber\\
			& & \quad +\int_{\Sigma}\langle \mathcal H\psi,\psi\rangle d\Sigma,  	
		\end{eqnarray*}
		where $\mathcal R$ and $\mathcal H$ are certain self-adjoint endomorphims acting on spinors. 
		Since the DECs (\ref{dom:energ:eq:0}) and (\ref{dom:energ:bd}) imply that $\mathcal R\geq 0$ and $\mathcal H\geq 0$, respectively, the mass inequality  $\mathscr E\geq |\mathscr P|_\delta$ follows after choosing $\phi$ properly. The rigidity statement is a bit more involved as we are supposed to start with the equality $\mathscr E=|\mathscr P|_\delta$ and then to prove that $\mathscr E=|\mathscr P|_\delta=0$ indeed. Once this is accomplished, we are able conclude the argument by means of a reflection
		method that allows us to apply the rigidity part of the PMT in \cite{huang2019equality} for manifolds without boundary. This last step employs in a crucial way an Ashtekar-Hansen-type formula for $\mathscr E$ \cite{de2019mass}.
	\end{proof}
	
	As it is apparent from this proof, Theorem \ref{main} holds true in any dimension $n\geq 3$ if we assume that $M$ is spin \cite{almaraz2014positive}. Note also that in the time-symmetric case ($\Pi=0$), we check that $\mathcal R=R_g/4$,  $\mathcal H=H_g/2$ and $\mathscr P=0$, so Theorem \ref{main} reduces to  a purely Riemannian assertion.
	
	\begin{theorem}\label{main:abl}\cite{almaraz2014positive}
		If $(M,g,\Sigma)$ is asymptotically flat as in Definition \ref{asym:flat:Pi} with $R_g\geq 0$ and $H_g\geq 0$ then its mass $\mathfrak m:=\mathscr E$ is nonnegative and vanishes only if $(M,g,\Sigma)=(\mathbb R^3_+,\delta^+,\partial \mathbb R^3_+)$ isometrically.
	\end{theorem}
	
	The following rigidity result, which is the analogue of Geroch's conjecture in our setting, is worth mentioning.
	
	\begin{corollary}\label{rig:pmt}
		Under the conditions of Theorem \ref{main:abl}, if $g=\delta^+$ outside a compact set then $(M,g,\Sigma)=(\mathbb R^3_+,\delta^+,\partial \mathbb R^3_+)$ isometrically.
	\end{corollary}
	
	We remark that the proof of Theorem \ref{main:abl} in \cite{almaraz2014positive} involves a doubling argument to reduce it to the standard (time-symmetric) PMT. Thus, in view of recent breakthroughs \cite{schoen2017positive,lohkamp2016higher} it actually holds true in any dimension (with no need for the spin assumption). 
	We remark that an alternative approach in low dimensions, based on the theory
	of free boundary minimal hypersurfaces, is presented in \cite{chai2018positive}.
	Also, it should be emphasized that the mean convexity condition $H_g\geq 0$ is derived as an immediate consequence of the boundary DEC, thus acquiring, likewise the interior DEC $R_g\geq 0$,  a justification on purely {\em physical} grounds: (\ref{dom:energ:bd}) means that the $3$-vector $S_{0a}$ is causal and future directed.

	\section{The case of a nonzero cosmological constant}\label{non:zero:lamb}
	
	If we now work in an arbitrary spatial dimension $n\geq 3$ and in the actions (\ref{action:eh}) and (\ref{gibb:haw}) we replace $R_{\overline g}$ by $R_{\overline g}-2\Lambda_{n,\varepsilon}$, where $\Lambda_{n,\epsilon}=\epsilon n(n-1)/2$, $\epsilon=\pm 1$, is a cosmological constant,  then 
	$G_{\overline g}$ gets replaced by $G_{\overline g}+\Lambda_{n,\epsilon}\overline g$ in the field equations (\ref{field:eq}) and (\ref{efe}). The model spacelike  metric is now given by 
	\begin{equation}\label{hads}
		\mathfrak g_{m,\epsilon}=\frac{ds^2}{1-\epsilon s^2-\frac{2m}{s^{n-2}}}+s^2\mathfrak h_0,
	\end{equation}
	where $m\geq 0$ is a real parameter, $\mathfrak h_0$ is the standard metric on $\mathbb S^{n-1}$ or $\mathbb S^{n-1}_+$ and $s$ varies in a certain interval $I_{m,\epsilon}\subset [0,+\infty]$.
	One easily verifies that (\ref{hads}) is the vacuum time-symmetric IDS associated to the Lorentzian metric 
	\[
	{{\mathfrak G}_{m,\epsilon}}=-\left(1-\epsilon s^2-\frac{2m}{s^{n-2}}\right)dx_0^2+\mathfrak g_{m,\epsilon},\quad x_0\in\mathbb R, 
	\] 
	which satisfies $G_{{\mathfrak G}_{m,\epsilon}}+\Lambda_{n,\epsilon}{{\mathfrak G}_{m,\epsilon}}=0$.
	We next consider rigidity/flexibility phenomena, in the spirit of Corollary \ref{rig:pmt}, where these models appear prominently.
	
	\begin{remark}\label{isot:coord}
		The case $\epsilon=0$ in (\ref{hads}) corresponds to the famous {\em Schwarzschild metric}, which models an uncharged, non-spinning black hole in GR with a (minimal) horizon located at $s=(2m)^{\frac{1}{n-2}}$. Another way of expressing this metric involves introducing a new radial parameter $r$ by 
		$$
		s={r}\left(1+\frac{m}{2 {r}^{n-2}}\right)^{\frac{2}{n-2}},
		$$
		so that a direct computation gives 
		\begin{equation}\label{isot:sch}
			\mathfrak{g}_{m,0}=\left(1+\frac{m}{2 {r}^{n-2}}\right)^{\frac{4}{n-2}}\delta.
		\end{equation}
		In these {\em isotropic} coordinates, the horizon is located at $r=(m/2)^{\frac{1}{n-2}}$. 
	\end{remark}
	
	\subsection{The asymptotically hyperbolic case ($\epsilon=-1$)}\label{neg:lamb}
	
	Due partially to potential applications to the so-called ADS/CFT correspondence, in recent years there has been a lot of interest in establishing positive mass theorems in case the asymptotic geometry at spatial infinity of the given spacetime is anti-deSitter space (or a quotient theoreof) \cite{biquard2005ads}; this corresponds to choosing $\epsilon=-1$ in (\ref{hads}). A notable result in this direction was put forward by Wang \cite{wang2001mass}, who worked in the spin category and treated the conformally compact case. In particular, the rigidity statement in his main theorem recovers a previous result  characterizing the standard hyperbolic space as the unique conformally compact, asymptotically hyperbolic spin manifold which is Einstein and whose conformal boundary at infinity is the canonical conformal structure on the sphere \cite{andersson1998scalar}. 
	Since our chief interest in this section is to explain how to extend this result in the presence of a noncompact boundary, we refrain from discussing the general theory of large scale invariants for asymptotically hyperbolic manifolds.  In any case, a good summary of recent results in this area may be found in \cite{huang2019mass}.

	Let $\overline N$ be a compact $n$-manifold, $n\geq 3$,  
	whose boundary $\partial \overline N$ decomposes as the union of two   smooth hypersurfaces 
	${\mathcal S}$ and $\Sigma$, with ${\mathcal S}$ being connected and such that  ${\mathcal S}\cap\Sigma$ is a $(n-2)$-dimensional ``corner''.
	If  $\mathfrak g$ is  a Riemannian metric on $N:={\rm int}\,\overline N\cup\Sigma$ then we say that  
	$(N,\mathfrak g)$ is {\em conformally compact} if there exists a collar neighborhood $\mathcal U\subset \overline N$ of $S$ such that on ${\rm int}\,\,\mathcal U$ we may write $\mathfrak g=\rho^{-2}{\widehat{\mathfrak g}}$ with $\widehat{\mathfrak g}$ extending to a sufficiently regular metric on $\mathcal U$ so that ${\mathcal S}$ and $\Sigma$ meet orthogonally (with respect to $\widehat{\mathfrak g}$) along their common boundary ${\mathcal S}\cap \Sigma$, where $\rho:\mathcal U\to \mathbb R$ is a {\em defining function} for ${\mathcal S}$ in the sense that  $\rho\geq 0$, $\rho^{-1}(0)={\mathcal S}$, $d\rho|_{\mathcal S}\neq 0$ and $\nabla_{\widehat{\mathfrak g}}\rho$ is tangent to $\Sigma$ along $\mathcal U\cap\Sigma$. 
	Clearly, the restriction $\widehat{\mathfrak g}|_{\mathcal S}$ determines a metric which changes by a conformal factor in case the defining function is changed.
	Thus, the conformal class $[\widehat{\mathfrak g}|_{\mathcal S}]$ of $\widehat{\mathfrak g}|_{\mathcal S}$ is well defined. We then say that the pair $(\mathcal S,[\widehat{\mathfrak g}|_{\mathcal S}])$ is the {\em conformal infinity} of $(N,\mathfrak g)$. 
	
	One then computes, as $\rho\to 0$, that the curvature tensor of ${\mathfrak g}$ is given by 
	$$
	R_{i j k \ell}=-|d \rho|_{\widehat{\mathfrak g}}^{2}\left(\mathfrak g_{i k} \mathfrak g_{j \ell}-\mathfrak g_{i \ell} \mathfrak g_{j k}\right)+{O}(\rho).
	$$
	Thus, 
	if $|d\rho|_{\widehat{\mathfrak g}}=1$ along ${\mathcal S}$ then 
	$(N,\mathfrak g)$ is {\em weakly} asymptotically hyperbolic in the sense that its 
	sectional curvature converges to $-1$ as one approaches $\mathcal S$. In this case, it is shown that  if $\mathfrak h_0$ is a metric on ${\mathcal S}$ representing  the given conformal infinity then there exists a unique defining function $\theta$ in $\mathcal U$ so that
	\begin{equation}\label{expan}
		\mathfrak g=\sinh^{-2}\theta\left(d\theta^2+\mathfrak h_{\theta}\right),
	\end{equation}
	where $\mathfrak h_{\theta}$ is a $\theta$-dependent family of metrics on ${\mathcal S}$ with $\mathfrak h_{\theta}|_{\theta=0}=\mathfrak h_0$.

	\begin{definition}\label{defasym2} 
		Let 
		$(N,\mathfrak g)$ be a weakly asymptotically hyperbolic manifold satisfying \eqref{expan}. We say that $(N,\mathfrak g)$ is {\em asymptotically hyperbolic} (in the conformally compact sense and with a non-compact boundary $\Sigma$) if its conformal infinity is $(\mathbb S_{+}^{n-1},[\mathfrak h_0])$, where $\mathfrak h_0$ is a round metric on $\mathbb S_{+}^{n-1}$, the unit upper $(n-1)$-hemisphere,
		and the following asymptotic expansion holds as $\theta\to 0$:
		\[
		\mathfrak h_{\theta}=\mathfrak h_0+\frac{\theta^n}{n!}\mathfrak h+\mathfrak k,
		\]
		where $\mathfrak h$ and $\mathfrak k$ are symmetric $2$-tensors on $\mathbb S^{n-1}_{+}$ and the remainder term $\mathfrak k$
		satisfies 
		\[
		|\mathfrak k|+|\nabla_{\mathfrak h_0} \mathfrak k|+|\nabla_{\mathfrak h_0}^2\mathfrak k|=o(\theta^{n+1}).
		\]
	\end{definition} 
	
	\begin{remark}\label{ads:model}
		It is instructive to briefly discuss the model space underlying Definition \ref{defasym2}. For this we fix a constant $m\geq 0$ and consider the metric in (\ref{hads}) with $\epsilon=-1$, 
		which is defined on $\mathbb H^{n}_{+,m}:=(s_{m,-1},+\infty)\times \mathbb S_+^{n-1}$, where $s_{0,-1}=0$ and for $m>0$, $s_{m,-1}>0$ is the unique positive zero of the denominator in (\ref{hads}). Clearly, $(\mathbb H^{n}_{+,0},\mathfrak g_{0,-1})$ is the {\em hyperbolic half-space}, which is obtained from the standard hyperbolic space by cutting along the totally geodesic hypersurface $(0,+\infty)\times \mathbb S^{n-2}$. For $m>0$, $(\mathbb H^{n}_{+,m},\mathfrak g_{m,-1})$ is the {\em anti-de Sitter-Schwarzschild (adSS) half-space}. These latter metrics satisfy $R_{\mathfrak g_{m,-1}}=-n(n-1)$ and  are asymptotically hyperbolic in the sense that $\mathfrak g_{m,-1}=\mathfrak g_{0,-1}+O(s^{-n})$ as $s\to +\infty$.
		In order to check that each $\mathfrak g_{m,-1}$ is conformally compact we rewrite it as 
		\[
		\mathfrak g_{m,-1}=\sinh^{-2}\theta(d\theta^2+u(\theta)^2\mathfrak h_0),
		\]	
		so that $s$ and $u$ satisfy the ODEs
		\[
		\dot s(\theta)  =  -\sinh^{-1}\theta\,\sqrt{1+s^2-\frac{2m}{s^{n-2}}}
		\]
		and 
		\[
		\cosh \theta\, u -\sinh \theta\,\dot u  = \sqrt{\sinh^2 \theta+u^2-2m\frac{\sinh^{n} \theta}{u^{n-2}}},
		\]
		with $s(0)=+\infty$ and $u(0) = 1$. 
		A short computation then shows that, as $\theta\to 0$, $u$ expands as 
		\[
		u(\theta)=1+c_n{m}\theta^n+o(\theta^{n+1}),\quad c_n>0,
		\]	
		as desired.
	\end{remark}

	The parameter $m$ may be viewed as the total mass of the gravitational system modeled by the IDS $(\mathbb H^{n}_{+,m},\mathfrak g_{m,-1})$. 
	More generally, a notion of mass may be assigned to any manifold as in Definition \ref{defasym2}, which, as in the asymptotically flat case, captures the rate of convergence of $\mathfrak g$ towards the reference metric $\mathfrak g_{0,-1}$, with the corresponding PMT being established in the spin category \cite{almaraz2020mass}\footnote{See also \cite{almaraz2021spacetime} for a non-time-symmetric version of this result.}. As a consequence of its rigidity statement we obtain the next theorem, which extends to our setting a previous result due to Andersson-Dahl \cite{andersson1998scalar}.

	\begin{theorem}\label{rigconfcom}\cite{almaraz2020mass}
		Let $(M,g,\Sigma)$ be a conformally compact, asymptotically hyperbolic spin $n$-manifold as above. Assume further that $g$ is Einstein 
		and that the mean curvature of $\Sigma$ is everywhere nonnegative. Then $(M,g,\Sigma)=(\mathbb H^n_{+,0},\mathfrak g_{0,-1},\partial \mathbb H^n_{+,0})$ isometrically. 
	\end{theorem}

	In words: the only way to fill in the conformal class  $(\mathbb S^{n-1}_+,[\mathfrak h_0])$ by an Einstein metric in the spin category and assuming further that the inner boundary is mean convex is the obvious
	one given by the hyperbolic half-space.
	This rigidity result may be of some interest in connection with recent developments involving the construction of a holographic dual to a conformal field theory defined on a manifold with boundary, the so-called ADS/BCFT correspondence\footnote{The ``B'' in BCFT means that the conformal boundary  itself carries a nontrivial boundary.}  \cite{takayanagi2011holographic,fujita2011aspects,nozaki2012central}.
	In this context, 
	the problem of determining the structure of the moduli space of conformally compact, Einstein  manifolds with a given conformal infinity and having a minimal inner boundary  plays a key role. A further development in this circle of ideas led to the formulation of the so-called {\em Ryu-Takayanagi conjecture} in Quantum Gravity \cite{ryu2006aspects}, which turns out to be a  huge generalization of the celebrated Beckenstein-Hawking formula for the black hole entropy and  proposes to compute the {\em entanglement entropy} determined by a proper region $\mathcal S\subset W$ in a given conformal infinity $(W,[\gamma])$, where $W$ is closed and $[\gamma]$ is the conformal class of a metric $\gamma$ on $W$. In the static case the conjecture roughly says that this kind of entropy is proportional to the area of the minimal inner boundary $\Sigma_{\mathcal S}$ which, together with $\mathcal S$, bounds a conformally compact Einstein region $(M,g)$ which has $(\mathcal S,[\gamma|_{\mathcal S}])$ as its conformal infinity. Theorem \ref{rigconfcom} suggests that, at least for conformal infinities close to to the conformal class  $(\mathbb S^{n-1}_+,[\mathfrak h_0])$ and for bulk regions parametrized by the half $n$-disk, the corresponding moduli space should be constituted by a unique configuration, thus eliminating any ambiguity in the choice of $\Sigma_{\mathcal S}$. This would extend to this setting a recent breakthrough in the boundaryless case \cite{chang2020compactness}.
	
	\begin{remark}\label{pen:hyp}
		The metric $\mathfrak g_{m,-1}$, $m>0$, is {\em not} a model for Theorem \ref{rigconfcom} for at least two reasons: it is not Einstein and, more importantly, it carries an inner minimal horizon (black hole) given by $s=s_{m,-1}$. In fact, this family of metrics models the equality case of a (conjectured) Penrose inequality in the asymptotically hyperbolic setting which is still wide open to the best of our knowledge. For instance, in the boundaryless case this inequality has only been proved in two situations: graphs in any dimension \cite{de2016alexandrov} and small perturbations of the adSS metric for $n=3$ \cite{ambrozio2014perturbations}. It would be interesting to investigate if the techniques in these works could be adapted in the presence of a boundary. 
	\end{remark}

	\begin{remark}\label{another:rig}
		It also follows from the main result in \cite{almaraz2020mass} that the conclusion of Theorem \ref{rigconfcom} still holds if we replace the Einstein assumption by two other natural requirements: $R_{g}\geq -n(n-1)$ everywhere and $g=\mathfrak g_{0,-1}$ outside a compact set. 
	\end{remark}
	
	\subsection{The case of positive cosmological constant ($\epsilon=1$)}\label{pos:cosm}
	
	Here we consider the model metric in (\ref{hads}) with $\epsilon=1$. This metric, named after de Sitter-Schwarzschild (dSS), satisfies $R_{\mathfrak g_{m,1}}=n(n-1)$ and for each $m$ such that
	\begin{equation}\label{mass:posv}
		0<m<m_*:=\frac{(n-2)^{(n-2) / 2}}{n^{n / 2}},
	\end{equation}  
	it is defined for $s$ varying in a certain bounded interval $(s_{m,1},s^\bullet_{m,1})\subset (0,+\infty)$. As $m\to 0$, $s_{m,1}\to 0$ and $s^\bullet_{m,1}\to 1$, so we recover the round hemisphere $(\mathbb S_+^{n}, \mathfrak h_0)$ in the limit. Even though the lack of an asymptotic region prevents us from defining a notion of mass in this setting, we may still ask whether  the dSS space is rigid with respect to a metric perturbation $g$ which coincides with $\mathfrak g_{m,1}$ up to first order along the totally geodesic boundary and satisfies the energy condition $R_g\geq n(n-1)$. Surprisingly enough, the next result confirms that this is never the case.

	\begin{theorem}\label{minoo:m}\cite{cruz2018deforming}
		For each $n \geq 3$ and for each value of the mass parameter as in (\ref{mass:posv}), the dSS space $M_{m,1}:=(s_{m,1},s^\bullet_{m,1})\times \mathbb S^{n-1}$, carries  metrics $\widetilde{g}_1$ and $\widetilde{g}_2$ with the following properties:
		\begin{itemize}
			\item${R}_{\widetilde {g}_1}>n(n-1)$ everywhere;
			\item $\widetilde {g}_1=\mathfrak g_{m,1}$ along $\partial M_{m,1}$;
			\item $\partial M_{m,1}$ is totally geodesic with respect to $\widetilde {g}_1$,
		\end{itemize}
		and 
		\begin{itemize}
			\item${R}_{\widetilde {g}_2}\geq n(n-1)$, with the strict inequality holding somewhere;
			\item $\widetilde {g}_2=\mathfrak g_{m,1}$ in a whole neighborhood of $\partial M_{m,1}$.
		\end{itemize}
	\end{theorem}

	As $m\to 0$ we recover the celebrated (negative) solution of Min-Oo's conjecture \cite{brendle2011deformations}.  
	Hence, flexibility seems to be a characteristic feature in the
	positive cosmological constant regime. Taken together, these results imply that there are no analogues of the rigidity statements of the positive mass and  
	Penrose inequality in this case; this should be contrasted to Corollary \ref{rig:pmt} and Remark \ref{another:rig}, where rigidity prevails thanks to the appropriate PMT.
	
	\begin{remark}\label{rem:final}
		As $m \rightarrow m_{*}$, the dSS space converges to the cylinder $\left[0, s_{*} \pi\right] \times \mathbb S^{n-1}({s^{*}})$ endowed with the metric $g_{m_{*},1}=d r^{2}+s_{*}^{2} \mathfrak h_0$, where 
		$s_*^2=(n-2)/n$. Unfortunately, the argument leading to Theorem \ref{minoo:m} above completely breaks down in this limit. It is an interesting question to examine the rigidity/flexibility of this IDS. Another nice question appears after consideration of the {\em half dSS space}  $M_{m,1,+}:=(s_{m,1},s^\bullet_{m,1})\times \mathbb S_+^{n-1}$, which carries the totally geodesic inner boundary $M_{m,1,+}:=(s_{m,1},s^\bullet_{m,1})\times \mathbb S^{n-2}$. Again, it would be interesting to investigate the flexibility of this space with an eye towards extending Theorem \ref{minoo:m} to this setting. 
	\end{remark}
	
	\section{Mass, center of mass and isoperimetry in the presence of a boundary}\label{center:sec}

	Here we bring the center of mass $\mathscr C$ defined in Section \ref{cons:qt:bd} to the forefront of our discussion on rigidity phenomena and explain how it plays a central role in determining the large scale isoperimetric profile of an asymptotically flat $3$-manifold $(M,g,\Sigma)$ of positive mass. The corresponding story in the boundaryless case is summarized in \cite{huang2009center,eichmair2013unique}. 
	
	For the sake of motivation, we start by fixing an interior point $q\in M$. For each $\mathsf r>0$ small enough we may consider the isoperimetric quotient 
	\[
	I_{\mathsf r}^M(q)=\frac{A_{\mathsf r}(q)^{3/2}}{V_{\mathsf r}(q)},
	\]
	where $A_{\mathsf r}(q)$, respectively $V_{\mathsf r}(q)$, is the area of the geodesic sphere, respectively the volume of the geodesic ball, of radius ${\mathsf r}$ centered at $q$. If 
	$I^{\mathbb R^3}=6\pi^{1/2}$ is the corresponding quotient in $\mathbb R^3$, which obviously does not depend on $(q,\mathsf r)$,  then a classical computation\footnote{Curiously enough, this isoperimetric interpretation of the scalar curvature is due to the same H. Vermeil mentioned in the beginning of Section \ref{class} and already appears in Pauli's famous encyclopedia article on GR \cite{pauli2relativitatstheorie}.} gives, as $\mathsf r\to 0$,
	\[
	1-\frac{I_{\mathsf r}^M(q)}{I^{\mathbb R^3}}=\frac{R_g(q)}{20}\mathsf r^2+O(\mathsf r^4).
	\] 
	Notice that the left-hand side may be computed if the metric $g$ is only assumed to be $C^0$. Thus, the validity of the local subisoperimetry property
	\begin{equation}\label{subiso:in}
		I_{\mathsf r}^M(q)\leq I^{\mathbb R^3},
	\end{equation}
	may be interpreted as saying that $R_g(q)\geq 0$ in a weak sense. Interestingly enough, a similar interpretation holds for the mean convexity condition $H_g(q)\geq 0$ at a boundary point. Indeed, after introducing Fermi coordinates around $q\in\Sigma$ we may consider the isoperimetric quotient
	\[
	I_{\mathsf r}^{M,\Sigma}(q)=\frac{A^+_{\mathsf r}(q)^{3/2}}{V^+_{\mathsf r}(q)},
	\]		
	where $A^+_{\mathsf r}(q)$ is the area of the coordinate hemisphere of radius ${\mathsf r}$ centered at $q$ and $V^+_{\mathsf r}(q)$ is the volume of the region enclosed by this hemisphere and $\Sigma$. This time we have 
	\[
	1-\frac{I_{\mathsf r}^{M,\Sigma}(q)}{I^{\mathbb R^3_+,\mathbb R^2}}=\frac{3}{16}{H_g(q)}\mathsf r+O(\mathsf r^2), 
	\]
	where $I^{\mathbb R^3_+,\mathbb R^2}=3\cdot 2^{1/2}\cdot\pi^{1/2}$ is  the corresponding quotient evaluated at hemispheres centered at $\mathbb R^2=\partial\mathbb R^3_+$.
	Thus, we may think of the local boundary subisoperimetry property
	\begin{equation}\label{subiso:in:bd}
		I_{\mathsf r}^{M,\Sigma}(q)\leq I^{\mathbb R^3_+,\mathbb R^2}
	\end{equation} 
	as expressing mean convexity of $\Sigma$ at $q$ in a weak sense. We may now envisage a far-reaching generalization of Theorem \ref{main:abl}.
	
	\begin{question}\label{far:reach:abl}
		If $(M,g,\Sigma)$ is a $C^0$ asymptotically flat $3$-manifold satisfying (\ref{subiso:in}) in the interior and (\ref{subiso:in:bd}) along the boundary, is it true that its ``isoperimetric mass'' is nonnegative and vanishes only if $(M,g,\Sigma)=(\mathbb R^3_+,\delta^+,\mathbb R^2)$ isometrically? 
	\end{question}
	
	As it is always the case with a (perhaps too) optimistic proposal, the proper definition of the concepts involved is part of the problem. However, a little examination of the smooth case reveals a natural candidate for the isoperimetric mass. 
	
	\begin{definition}\label{iso:mass:sm}
		For a (smooth) asymptotically flat $3$-manifold $(M,g,\Sigma)$ as in Definition \ref{asym:flat:Pi}, let 
		$A^+(r)$ (respectively $V^+(r)$) be the area of $S^2_{r,+}$ (respectively, the volume of the compact region enclosed by $S^2_{r,+}$ and  $\Sigma$) and set
		\[
		\mathcal I_r^{M;\Sigma}
		=\frac{V^+(r)}{A^+(r)}\left(1-{\frac{{I}^{M;\Sigma}_r}{I^{\mathbb R^3_+;\mathbb R^2}}}\right),
		\]
		where 
		\[
		{I}^{M;\Sigma}_r=\frac{A^+(r)^{\frac{3}{2}}}{V^+(r)},
		\]
		and $I^{\mathbb R^3_+;\mathbb R^2}=3\cdot 2^{1/2}\pi^{1/2}$ is the corresponding isoperimetric quotient computed at a hemisphere centered at a point in $\mathbb R^2=\partial \mathbb R^3_+$. 
		With this notation, the  {\em isoperimetric mass} of $(M,g,\Sigma)$ is defined  by 
		\[
		\mathfrak m^{\rm iso}=\lim_{r\to +\infty}\mathcal I^{M;\Sigma}_r.
		\]
	\end{definition} 
	
	That we follow a promising route is manifested by the next result, which extends to our setting a famous remark by Huisken \cite{huisken2006isoperimetric}.
	
	\begin{theorem}\label{iso=mass}\cite{almaraz2020center}
		One has $\mathfrak m^{\rm iso}=\mathfrak m$. 
	\end{theorem}

	\begin{corollary}\label{iso=mass:col}
		If $(M,g,,\Sigma)$ satisfies $R_g\geq 0$ in the interior and $H_g\geq 0$ along the boundary then $\mathfrak m^{\rm iso}\geq 0$, with the strict inequality holding unless $(M,g,\Sigma)=(\mathbb R^3_+,\delta^+,\mathbb R^2)$ isometrically. 
	\end{corollary}
	
	\begin{proof}
		Apply Theorem \ref{main:abl}.
	\end{proof}
	
	This corollary may be viewed as a manifestation  of the validity of Question \ref{far:reach:abl}
	in the smooth world, thus supplying some evidence for is validity in general. 
	
	After this somewhat lengthy preamble, which emphasized the relationship between the (sign of the) mass and the large scale isoperimetric properties of an IDS with a noncompact boundary, we now investigate what the center of mass has to say in this respect. For this it is convenient to restrict ourselves to a special class of IDSs, whose behavior at infinity is modeled on the Schwarzschild metric (\ref{isot:sch}) with $n=3$. 
	
	\begin{definition}\label{half:shc:def} 
		An asymptotically flat $3$-manifold with a noncompact boundary $(M,g,\Sigma)$
		is {\em asymptotically half-Schwarzschild}
		(ahS) if a neighborhood of infinity is diffeomorphic to the complement of a hemisphere in $\mathbb R^3_+$ so that 
		\[
		g=\left(1+\frac{2m}{r}\right)\delta^++p^+,\quad p^+= O(r^{-2}) 
		\]
		holds in this asymptotic region.
	\end{definition}
	
	As usual, if we further take it for granted that the corresponding Regge-Teitelboim conditions are met, then  the limit defining $\mathscr C$ in Table 2 converges  \cite{de2019mass}. Now, Theorem \ref{iso=mass} suggests that for an ahS manifold with $\mathfrak m>0$, large coordinate hemispheres may be perturbed to yield global solutions of the corresponding {\em relative} isoperimetric problem, where each competing surface $S$ satisfies $\partial S\subset\Sigma$ and ${\rm int}\,S\cap \Sigma=\emptyset$, with the 
	constrained  volume being the one  enclosed by $S$ and $\Sigma$. 
	The next result, which identifies $\mathscr C$ to the center of a geometric foliation at infinity, turns out to be a first step towards this goal. 
	
	\begin{theorem}\label{free:af:bd}\cite{almaraz2020center}
		Assume that $(M,g)$ is an ahS   $3$-manifold with a noncompact boundary $\Sigma$. If ${\mathfrak m}=m/2 >0$ then there exists a neighborhood of infinity which is foliated by   {strictly stable} free boundary contant mean curvature hemispheres.
		Moreover, the geometric center  of this  foliation coincides with the center of mass $\mathscr C$ of $(M,g,\Sigma)$.
	\end{theorem}
	
	The proof of this result makes use of the so-called implicit function method pioneered by Ye \cite{ye1997foliation} and refined by Huang \cite{huang2008center}. Our next contribution completely solves the relative isoperimetric problem referred to above by extending a celebrated result due to Eichmair-Metzger \cite{eichmairlarge} to our setting.

	\begin{theorem}\label{iso:ext:em1}\cite{almaraz2020center}
		Let $(M,g,\Sigma)$ be as in Theorem \ref{free:af:bd}. Then for all sufficiently large volume there exists an associated bounded, relative isoperimetric region whose connected and smooth  boundary remains close to a centered coordinate hemisphere, with the region sweeping out the whole manifold as the volume diverges towards infinity. In particular, the corresponding  isoperimetric surfaces coincide with the leaves of the foliation in Theorem \ref{free:af:bd}, thus being unique for each value of the enclosed volume.  
	\end{theorem} 
	
	This result goes a long way towards determining an asymptotic expansion for the relative isoperimetric profile of ahS $3$-manifolds with positive mass for all sufficiently large values of the enclosed volume.

			\bibliographystyle{plain}
			\bibliography{conserved-delima-arxiv-v1}

		\end{document}